\documentclass{amsart}

\usepackage[english]{babel}
\usepackage[utf8]{inputenc}
\usepackage[colorinlistoftodos]{todonotes}
\usepackage{esint}
\usepackage{mathrsfs}
\usepackage{breqn}
\usepackage{graphics}
\usepackage{enumerate}
\usepackage{graphicx}
\usepackage{placeins}
\usepackage{amsmath,amsthm,amsfonts,amssymb,euscript,verbatim,bbm}
\usepackage{hyperref}
\usepackage{amsrefs}

\usepackage{color}

\newtheorem{thm}{Theorem}
\newtheorem{remark}[thm]{Remark}
\newtheorem{prop}[thm]{Proposition}
\newtheorem{cor}[thm]{Corollary}
\newtheorem{lemma}[thm]{Lemma}

\newcommand{\eqdef}{\overset{\mbox{\tiny{def}}}{=}}

\def\Ddim {d}
\def\ss {\mu}
\def\rr {\nu}

\newcommand{\ba}{\begin{equation}}
\newcommand{\ea}{\end{equation}}

\newcommand{\bea}{\begin{eqnarray}}
\newcommand{\eea}{\end{eqnarray}}

\def\beaa{\begin{eqnarray*}}
\def\eeaa{\end{eqnarray*}}

\begin{document}

\title[Decay Estimates for the Muskat Equation]{Large Time Decay Estimates for the Muskat Equation}

\author{Neel Patel}
\address{Department of Mathematics, University of Pennsylvania, Philadelphia, PA 19104, USA.}
\email{neelpa@sas.upenn.edu}

\author{Robert M. Strain}
\address{Department of Mathematics, University of Pennsylvania, Philadelphia, PA 19104, USA.}
\email{strain@math.upenn.edu}
\thanks{R.M.S. was partially supported by the NSF grant DMS-1200747.}
\date{\today; 
}

\begin{abstract}
We prove time decay of solutions to the Muskat equation in 2D and in 3D. In \cite{JEMS} and \cite{CCGRPS}, the authors introduce the norms 
$$
\|f\|_{s}(t)\eqdef   \int_{\mathbb{R}^{2}} |\xi|^{s}|\hat{f}(\xi)| \ d\xi
$$ 
in order to prove global existence of solutions to the Muskat problem. In this paper, for the 3D Muskat problem, given initial data $f_{0}\in H^{l}(\mathbb{R}^{2})$ for some $l\geq 3$ such that $\|f_{0}\|_{1} < k_{0}$ for a constant $k_{0} \approx 1/5$, we prove uniform in time bounds of $\|f\|_{s}(t)$ for $-d < s < l-1$ and assuming $\|f_{0}\|_{\nu} < \infty$ we prove time decay estimates of the form 
$\|f\|_{s}(t) \lesssim (1+t)^{-s+\nu}$ for $0 \leq s \leq l-1$ and $-d \leq \nu < s$.  These large time decay rates are the same as the optimal rate for the linear Muskat equation.  We also prove analogous results in 2D.
\end{abstract}

\setcounter{tocdepth}{1}

\maketitle
\tableofcontents

\section{Introduction}

The Muskat problem describes the dynamics between two incompressible immiscible fluids in porous media such that the fluids are of different constant densities.   The Muskat problem is an extensively studied well established problem \cite{MR2070208,MR0097227,muskat,MR1294935,HeleShawN,GraneroBelinchon,MR3181769,MR3466160,MR3181767,Escher,MR2472040,interface,MR2753607,MR1223740,CGSV,MR1668586,CCGRPS,JEMS,ADP2,ChGBSh,MR3482335,CCFG,CCFGL,Beck,Bear,MR1612026,MR2128613}.  In this paper we consider the interface between the two fluids under the assumption that there is no surface tension and the fluids are of the same constant viscosity. Because the fluids are immiscible, we can assume that we have a sharp interface between the two fluids. Without loss of generality we normalize gravity $g=1$, permeability $\kappa =1$ and viscosity $\nu =1$.  Then the 3D Muskat problem is given by
\begin{align}\label{Muskat}
\rho_{t} + \nabla\cdot(u\rho) &= 0 \\
u+\nabla P &= -(0,0,\rho) \\
\nabla\cdot u &= 0
\end{align}
where $\rho = \rho(x_{1},x_{2},x_{3},t)$ is the fluid density function, $P= P(x_{1},x_{2},x_{3},t)$ is the pressure, and
$u = (u_{1}(x_{1},x_{2},x_{3},t), u_{2}(x_{1},x_{2},x_{3},t), u_{3}(x_{1},x_{2},x_{3},t))$
is the incompressible velocity field. Here $x_{i} \in\mathbb{R}$ for $i=1,2,3$ and $t\geq 0$. The third equation of this system simply states that the fluids are incompressible.  Given the incompressibility condition, the first equation is a conservation of mass equation, as the fluid density is preserved along the characteristic curves given by the velocity field. The second equation is called Darcy's Law, which governs the flow of a fluid through porous medium. When we assume that the two incompressible fluids are  of constant density, then the function $\rho(x_{1},x_{2},x_{3},t)$ can be written as
\[ \rho(x_{1},x_{2},x_{3},t) = 
\begin{cases} 
	\rho^{1} & (x_{1},x_{2},x_{3})\in \Omega^{1}(t) = \{x_{3} > f(x_{1},x_{2},t)\} \\
	\rho^{2} & (x_{1},x_{2},x_{3})\in \Omega^{2}(t) = \{x_{3} < f(x_{1},x_{2},t)\}
\end{cases} \]
where $\Omega^{i}$ for $i = 1,2$ are the regions in $\mathbb{R}^{3}$ occupied by the fluids of density $\rho^{i}$ for $i=1,2$ respectively and the equation $x_{3}= f(x_{1},x_{2},t)$ describes the interface between the two fluids.   We consider the stable regime (see \cite{interface}) in which $\rho_{1} < \rho_{2}$. 

The interface function, $f: \mathbb{R}_{x}^{2}\times \mathbb{R}^{+}_{t} \rightarrow \mathbb{R}$ is known to satisfy the equation 
\bea\label{interfaceq}
\frac{\partial f}{\partial t}(x,t) = \frac{\rho^{2}-\rho^{1}}{4\pi} PV\int_{\mathbb{R}^{2}}\frac{(\nabla f(x,t)-\nabla f(x-y,t))\cdot y}{[|y|^{2}+(f(x,t)-f(x-y,t))^{2})]^{\frac{3}{2}}} \ dy
\eea
with initial data $f(x,0) = f_{0}(x)$ for $x = (x_{1},x_{2}) \in\mathbb{R}^{2}$.
Without loss of generality for the results in this paper we can take $\frac{\rho_{2}-\rho_{1}}{2} = 1$. Then, as given in \cite{CCGRPS}, the 3D Muskat interface equation can be written as
\bea\label{interface3D}
f_{t}(x,t) = -\Lambda f - N(f), 
\eea
where $\Lambda$ is the square root of the negative Laplacian and
\bea\label{Nf}
N(f)(x) = \frac{1}{2\pi}\int_{\mathbb{R}^{2}}\frac{|y|}{y^{2}}\cdot \nabla_{x}\triangle_{y}f(x)R(\triangle_{y}f(x)) \ dy,
\eea
where 
$$
R(t) = 1- \frac{1}{(1+t^{2})^{\frac{3}{2}}}
$$ 
and 
$$
\triangle_{y}f(x) = \frac{f(x)-f(x-y)}{|y|}.
$$
We will use equation \eqref{interface3D} to prove uniform in time norm bounds and large time decay rates for the solution $f(t,x)$ in 3D.

We will also prove uniform in time norm bounds and large time decay rates for the 2D Muskat problem.  The 2D Muskat problem is given by the interface equation
\bea\label{interfaceq2D}
\frac{\partial f}{\partial t}(x,t) = \frac{\rho^{2}-\rho^{1}}{4\pi}\int_{\mathbb{R}}\frac{(\nabla f(x,t)-\nabla f(x-\alpha,t))\alpha}{\alpha^{2}+(f(x,t)-f(x-\alpha,t))^{2}} \ d\alpha
\eea
with initial data $f(x,0) = f_{0}(x)$ for $x\in\mathbb{R}$. The density function $\rho$ is given by
\[ \rho(x_{1},x_{2},t) = 
\begin{cases} 
\rho^{1} & (x_{1},x_{2})\in \Omega^{1}(t) = \{x_{2} > f(x_{1},t)\} \\
\rho^{2} & (x_{1},x_{2})\in \Omega^{2}(t) = \{x_{2} < f(x_{1},t)\}
\end{cases}. \]
Similarly to above, we can rewrite the 2D interface equation setting $\frac{\rho_{2}-\rho_{1}}{2} = 1$, as given by (9) in \cite{JEMS}
\bea\label{interface2dlinear}
f_{t}(x,t) = -\Lambda f - T(f)
\eea
where
\bea\label{Tf}
T(f) = \frac{1}{\pi} \int_{\mathbb{R}}\triangle_{\alpha}\partial_{x}f(x) \frac{(\triangle_{\alpha}f(x))^{2}}{1+(\triangle_{\alpha}f(x))^{2}} \ d\alpha
\eea
and $$\triangle_{\alpha}f(x) = \frac{f(x)- f(x-\alpha)}{\alpha}.$$
The equations (\ref{interface3D}) and (\ref{interface2dlinear}) will be the relevant formulations of the interface equation that we will use in this paper.

\section{Main Results}

In this section we will first introduce our notation.  Then, we will state our main results and explain relevant prior results on the Muskat problem.  Following that we discuss our strategy for proving the large time decay rates and we give an outline the rest of the article.

\subsection{Notation}
Typically we have for the dimension that $\Ddim \in \{1,2\}$.
Then we consider the following norm introduced in \cite{JEMS}:
\bea \label{Snorm}
\|f\|_{s} \eqdef \int_{\mathbb{R}^{\Ddim}} |\xi|^{s}|\hat{f}(\xi)| \ d\xi,
\eea
where $\hat{f}$ is the standard Fourier transform of $f$:  
$$
\hat{f}(\xi) \eqdef  \mathcal{F}[f](\xi) = \int _{\mathbb{R}^{\Ddim}} f(x) e^{-2\pi i x\cdot \xi} dx.
$$
We will use this norm generally for $s>-\Ddim$ and we refer to it as the \textit{s-norm}.
To further study the case $s = -\Ddim$, then for $s\ge -\Ddim$ we define the  \textit{Besov-type s-norm}:
\bea \label{normSinfty}
\|f\|_{s,\infty} \eqdef \Big\|\int_{C_{j}} |\xi|^{s}|\hat{f}(\xi)| \ d\xi\Big\|_{l^{\infty}_{j}}
=
\sup_{j \in \mathbb{Z}} \int_{C_{j}} |\xi|^{s}|\hat{f}(\xi)| \ d\xi,
\eea
where $C_{j} = \{\xi \in \mathbb{R}^\Ddim : 2^{j-1}\leq |\xi| < 2^{j}\}$.  Note that we have the inequality
\begin{equation}\label{ineqbd}
\|f\|_{s,\infty}  \leq \int_{\mathbb{R}^{d}} |\xi|^{s}|\hat{f}(\xi)| \ d\xi = \|f\|_{s}.	
\end{equation}
We point out that $\|f\|_{-\Ddim/p,\infty} \lesssim \|f\|_{L^p(\mathbb{R}^\Ddim)}$ for $p\in [1,2]$ as is shown in Lemma \ref{lem.bounds.upper}.  This and other embeddings are established in Section \ref{sec.embedding}.

Next, consider the operator $|\nabla|^{r}$ defined  for $r \in \mathbb{R}$ by
$$ 
\widehat{|\nabla|^{r} f}(\xi) = |\xi|^{r} \hat{f}(\xi).
$$
The Sobolev norms on the homogeneous Sovolev spaces $\dot{W}^{r,p}(\mathbb{R}^\Ddim)$ and inhomogeneous Sobolev spaces $W^{r,p}(\mathbb{R}^\Ddim)$ for $r\in\mathbb{R}$ and $1\leq p \leq \infty$ are given by:
\bea
\|f\|_{\dot{W}^{r,p}} = \||\nabla|^{r}f\|_{L^{p}(\mathbb{R}^\Ddim)}.
\eea
and
\bea
\|f\|_{W^{r,p}} = \|(1+|\nabla|^{2})^{\frac{r}{2}}f\|_{L^{p}(\mathbb{R}^\Ddim)}.
\eea
In the special case $p=2$, we write $W^{r,2}(\mathbb{R}^\Ddim) = H^{r}(\mathbb{R}^\Ddim)$  and   $\dot{W}^{r,2}(\mathbb{R}^\Ddim) =\dot{H}^{r}(\mathbb{R}^\Ddim)$.  We define the convolution of two functions as usual as
$$
(f \ast g)(x) = \int_{\mathbb{R}^{d}} f(y) g(x-y) dy.
$$
We adopt the following convention for an iterated convolutions of the same function
$$
\ast^{n}f \eqdef f \ast f \ast \cdots \ast f
$$
where the left-hand side of the above is a convolution of the function $f$ $n$ times.  This notation will be useful in some of the estimates.

Finally, we use the notation $f_{1} \lesssim f_{2}$ if there exists a uniform constant $C > 0$, that does not depend upon time, such that $f_{1} \leq C f_{2}$.  Also $f_{1} \approx f_{2}$ means that $f_{1} \lesssim f_{2}$ and $f_{2} \lesssim f_{1}$.  

\smallskip

Having explained the necessary notation, we now state our main results.

\subsection{Main theorem}

Given a well-defined fluid interface that exists globally in time, we will study its long-time behavior. We will first use a known well-posedness theory to establish a setting in which to study long-time behavior. In Theorem 3.1, \cite{CCGRPS} has the following global existence result in 3D:

\begin{thm}\label{oldthm3D}
Suppose that $f_{0} \in H^{l}(\mathbb{R}^{2})$, for some $l\geq 3$, and $\|f_{0}\|_{1} < k_{0}$ where $k_{0}>0$ satisfies for some $0<\delta < 1$ that
\bea\label{k0}
\pi\sum_{n\geq 1} (2n+1)^{1+\delta}\frac{(2n+1)!}{(2^{n}n!)^{2}}k_{0}^{2n} \leq 1.
\eea
 Then there exists a unique solution $f$ of \eqref{interfaceq} with initial data $f_{0}$.  Furthermore $f\in C([0,T]; H^{l}(\mathbb{R}^{2}))$ for any $T > 0$.  Any $0\le k_{0} \leq \frac{1}{5}$ satisfies (\ref{k0}) for some $\delta > 0$.
\end{thm}

In the proof of Theorem \ref{oldthm3D}, the authors \cite{CCGRPS} show that $\|f\|_{1}$ is uniformly bounded in time. They first bound $\mathscr{F}(N(f))=\widehat{N(f)}$ as follows:
$$
\int_{\mathbb{R}^{2}} |\xi||\mathscr{F}(N(f))| \ d\xi \leq  \pi\Big(\frac{1+2\|f\|_{1}^{2}}{(1-\|f\|_{1}^{2})^{\frac{5}{2}}}-1\Big)\|f\|_{2}.
$$
Then using the inequality
\begin{equation}
\frac{d}{dt} \|f\|_{1} (t) \leq -\int_{\mathbb{R}^{2}} d\xi \ |\xi|^{2}|\hat{f}(\xi)|  \ + \int_{\mathbb{R}^{2}} d\xi \ |\xi||\mathscr{F}(N(f))(\xi)|,
\end{equation}
it is shown that
\bea \notag 
\frac{d}{dt}\|f\|_{1}(t) \leq \Big(\pi\Big(\frac{1+2\|f\|_{1}^{2}}{(1-\|f\|_{1}^{2})^{\frac{5}{2}}}-1\Big)-1\Big)\|f\|_{2}.
\eea
Further since
$$
\pi\Big(\frac{1+2k_{0}^{2}}{(1-k_{0}^{2})^{\frac{5}{2}}}-1\Big)-1 < 0,
$$ 
it is seen for some $C_{0}= C_{0}(\|f_{0}\|_{1}) > 0$ that
\bea\label{1diffbound}
\frac{d}{dt}\|f\|_{1}(t) \leq -C_{0}\|f\|_{2}.
\eea
 In particular it holds for all $t\geq 0$ that
$
\|f\|_{1}(t) \leq \|f_{0}\|_{1} < k_{0}.
$

Related existence results can be shown in 2D \cite{CCGRPS}:  

\begin{thm} \cite{JEMS,CCGRPS} \label{2Doldthm}
If $f_{0}\in H^{l}(\mathbb{R})$ for some $l\geq 2$ and $\|f_{0}\|_{1} < c_{0}$ where $c_{0}$ satisfies
\bea\label{c0}
2\sum_{n\geq 1} (2n+1)^{1+\delta}c_{0}^{2n} \leq 1
\eea
for some $0<\delta <\frac{1}{2}$, then there exists a unique global in time solution $f$ of the Muskat problem \eqref{interfaceq2D} in 2D with initial data $f_{0}$ such that $f\in C([0,T];H^{l}(\mathbb{R}))$ for any $T>0$.  Further \eqref{c0} holds if for example $0\le c_{0}\le 1/3$.
\end{thm}

Analogously, in the course of the proof of the 2D existence Theorem \ref{2Doldthm}, it is shown that
\bea\label{1normineq2D}
\frac{d}{dt}\|f\|_{1}(t) \leq -\beta \|f\|_{2}(t),
\eea
for  a constant $\beta>0$ depending on the $c_{0}$ and $\|f_0\|_{H^{2}(\mathbb{R}^{\Ddim})}$.  These differential inequalities \eqref{1diffbound}  and \eqref{1normineq2D} will be very useful for proving the time decay rates.

In this paper, we prove time-decay rates for solutions to the Muskat problem.   For simplicity we will state our main theorem so that it holds in either dimension $\Ddim \in \{1,2\}$.  We consider a solution to the Muskat problem satisfying all of the assumptions of Theorem \ref{oldthm3D} (when $\Ddim=2$) or Theorem \ref{2Doldthm} (when $\Ddim=1$).

\begin{thm}\label{mainthm}
Suppose $f$ is the solution to the Muskat problem either described by Theorem \ref{oldthm3D} in 3D \eqref{interfaceq} , or described by Theorem \ref{2Doldthm} in 2D \eqref{interfaceq2D}.   In this case the initial data satisfies $f_{0} \in H^{l}(\mathbb{R}^{\Ddim})$ for some $l\geq 1+\Ddim$.

Then, for $-\Ddim <  s < l-1$, we have the uniform in time estimate
\bea\label{main2}
\|f\|_{s}(t) \lesssim 1.
\eea
In addition for $0 \le  s < l-1$ we have the uniform time decay estimate
\bea\label{main}
\|f\|_{s}(t) \lesssim (1+t)^{-s+\rr},
\eea
where we allow $\rr$ to satisfy $-\Ddim \le \rr < s$.  

For \eqref{main}, when $\rr > -\Ddim$ then we require 
additionally that $\|f_{0}\|_{\rr} < \infty$, and when $\rr = -\Ddim$ then we alternatively require 
$\|f_{0}\|_{-\Ddim,\infty} < \infty$.  The implicit constants in \eqref{main2} and \eqref{main} depend on $\|f_0\|_{s}<\infty$ and $k_{0}$.  In \eqref{main} the implicit constant further depends on  either  $\|f_{0}\|_{\rr}$ (when $\rr > -\Ddim$) or $\|f_{0}\|_{-\Ddim,\infty}$ (when $\rr = -\Ddim$).
\end{thm}

It can be directly seen from the proof that for \eqref{main}, when $\rr > -\Ddim$ then one only needs to assume $\|f_{0}\|_{s,\infty} < \infty$ instead of the stronger condition $\|f_0\|_{s}<\infty$.   Also note that it is shown in Proposition \ref{endpointprop} and Section \ref{2Dsection} that $\|f\|_{-\Ddim,\infty}(t) \lesssim 1$ and we more generally have $\|f\|_{s,\infty}(t) \lesssim 1$ for $\rr \ge -\Ddim$ from \eqref{ineqbd}.

 The Muskat problem \eqref{interfaceq} or \eqref{interfaceq2D} can be linearized around the flat solution, which can be taken as $f(x,t)=0$, to find the following linearized  nonlocal partial differential equation
\begin{align}
\begin{split}\label{le}
f_t(x,t)&=-\frac{\rho^2-\rho^1}{2}\Lambda f(x,t),\\ 
f(\alpha,0)&=f_0(\alpha), \quad \alpha\in\mathbb{R}.
\end{split}
\end{align}
Here the operator $\Lambda$ is defined in Fourier variables by 
$\widehat{\Lambda f}(\xi)=|\xi|\widehat{f}(\xi)$.  This linearization shows the parabolic character of the Muskat problem in the stable case  which is $\rho^2 > \rho^1$ (\cite{interface}).

Notice that the decay rates which we obtain in Theorem \ref{mainthm} are consistent with the optimal large time decay rates for \eqref{le}.  In particular it can be shown by standard methods that if $g_0(x)$ is a tempered distribution vanishing at infinity and satisfying 
$\|g_{0}\|_{\rr,\infty} < \infty$, then   one further has
$$
\| g_0\|_{\rr,\infty}
\approx
\left\| t^{s-\rr} \left\|  e^{t \Lambda} g_0 \right\|_{s}  \right\|_{L^\infty_t((0, \infty) )}, \quad \text{for any $s\ge \rr$.}
$$
This equivalence then grants the optimal time decay rate of $t^{-s+\rr}$ for $\left\|  e^{t \Lambda} g_0 \right\|_{s}$ that is the same as the non-linear time decay in \eqref{main}.

Previously in 2009 in \cite{MR2472040} has shown that the Muskat problem satisfies a maximum principle $\|f\|_{L^{\infty}}(t)\leq \|f_0\|_{L^{\infty}}$; decay rates are obtained for the periodic case ($x \in \mathbb{T}^\Ddim$) as: 
 $$
 \|f\|_{L^\infty(\mathbb{T}^\Ddim)}(t)\leq \|f_0\|_{L^\infty(\mathbb{T}^\Ddim)}e^{-(\rho_2 - \rho_1)C(\|f_0\|_{L^\infty(\mathbb{T}^\Ddim)}) t},
 $$
 where the mean zero condition is used.    In the whole space case (when the interface is flat at infinity) then again in \cite{MR2472040} decay rates are obtained of the form 
$$
\|f\|_{L^\infty(\mathbb{R}^\Ddim)}(t)\leq \|f_0\|_{L^\infty(\mathbb{R}^\Ddim)}
\left(1+(\rho_2 - \rho_1)C(\|f_0\|_{L^\infty(\mathbb{R}^\Ddim)},\|f_0\|_{L^1(\mathbb{R}^\Ddim)})t\right)^{-\Ddim}.
$$ 
To prove this time decay in $\mathbb{R}^\Ddim$ they suppose that initially either $f_0(x) \ge 0$ or $f_0(x) \le 0$.  Notice that by the Hausdorff-Young inequality then \eqref{main} also proves this  $L^\infty(\mathbb{R}^\Ddim)$ decay rate of $t^{-\Ddim}$ under the condition $\|f_{0}\|_{-\Ddim,\infty} < \infty$.

Furthermore \cite{CCGRPS}, it is shown that if $\|\nabla f_0\|_{L^{\infty}(\mathbb{R}^2)} < 1/3$ then the solution of (\ref{interfaceq}) with initial data $f_{0}$ satisfies the uniform in time bound 
	$
	\|\nabla f\|_{L^{\infty}(\mathbb{R}^2)}(t) < 1/3.
	$
Note that \eqref{main} implies in particular when $\Ddim=2$ that 
\bea\label{gradientdecay}
\|\nabla f \|_{L^{\infty}_{x}} \lesssim \||\xi||\hat{f} |\|_{L^{1}_{\xi}} = \|f\|_{1} \lesssim (1+t)^{-3}.
\eea
However decay estimate (\ref{gradientdecay}) requires $\|f_{0}\|_{1} < k_{0}$ and $\|f_{0}\|_{-2,\infty} < \infty$, which is a stronger assumption than $\|\nabla f_{0}\|_{L^{\infty}} < 1/3$.

We further obtain the following corollary directly from the Hausdorff-Young inequality; this is explained in the embedding result (\ref{embedlower}) below.
\begin{cor}
Suppose $f$ is the solution to the Muskat problem either described by Theorem \ref{oldthm3D} in 3D \eqref{interfaceq} , or described by Theorem \ref{2Doldthm} in 2D \eqref{interfaceq2D}.   In this case the initial data satisfies $f_{0} \in H^{l}(\mathbb{R}^{\Ddim})$ for some $l\geq 1+\Ddim$.

Then, for $-\Ddim <  s < l-1$, we have the uniform in time estimate
\bea\label{main2.1}
\|f\|_{\dot{W}^{s,\infty}}(t) \lesssim 1.
\eea
In addition for $0 \le  s < l-1$ we have the uniform time decay estimate
\bea\label{main'}
\|f\|_{\dot{W}^{s,\infty}}(t) \lesssim (1+t)^{-s+\rr},
\eea
where we allow $\rr$ to satisfy $-\Ddim \le \rr < s$.  

For \eqref{main'}, when $\rr > -\Ddim$ then we require 
additionally that $\|f_{0}\|_{\rr} < \infty$, and when $\rr = -\Ddim$ then we alternatively require 
$\|f_{0}\|_{-\Ddim,\infty} < \infty$.  The implicit constant in \eqref{main2.1} and \eqref{main'} depend on $\|f_0\|_{s}<\infty$ and $k_{0}$.  In \eqref{main'} the implicit constant further depends on  either  $\|f_{0}\|_{\rr}$ (when $\rr > -\Ddim$) or $\|f_{0}\|_{-\Ddim,\infty}$ (when $\rr = -\Ddim$).
\end{cor}

Thus, under the assumptions of Theorem \ref{mainthm}, defining $\nabla^{\alpha}f \eqdef \partial_{x_{1}}^{\alpha_{1}}\partial_{x_{2}}^{\alpha_{2}} f$ where $\alpha = (\alpha_{1},\alpha_{2})$ and $|\alpha| = \alpha_{1}+\alpha_{2}$, we know that, up to order $|\alpha| < l-1$, the derivatives $\|\nabla^{\alpha}f\|_{L^{\infty}}$ decay in time with the optimal linear decay rate.

\subsection{Strategy of proof} We first explain the 3D Muskat problem.  Our strategy of this proof is two-fold. We will first prove uniform bounds on $\|f\|_{s}$ for $-\Ddim < s<2$ and $\|f\|_{s,\infty}$ for $-\Ddim \le s<2$ including $s = -\Ddim$. Then afterwards we use these uniform bounds to prove the large time decay for $0\leq s < l-1$.

To this end we prove an embedding lemma, which allows us to bound $\|f\|_{s}$ for $-1< s < 2$ as
$$ 
\|f\|_{s} \lesssim \|f\|_{H^{3}}.
$$
Since our interface solution $f(x,t)$ is uniformly bounded under the $H^{l}$ Sobolev norm for some $l\geq 3$, we obtain uniform bounds on $\|f\|_{s}(t)$ for $-1< s < 2$. Now, we can use (\ref{1diffbound}) and the general decay Lemma \ref{decaylemma} to obtain an initial decay result for $0\leq s \leq 1$:
$$ \|f\|_{s} \lesssim (1+t)^{-s+\nu}$$
where $-1 < \nu < s$ and the implicit constant depends on $\|f_{0}\|_{\nu}$. We then will make use of this decay inequality for $s=1$ to prove uniform bounds for the range $-2<s<1$ as follows.

First, we need an appropriate bound on the time derivative of $\|f\|_{s}(t)$. To this end, we have the differential inequality
$$  \frac{d}{dt} \|f\|_{s} (t) +C\int_{\mathbb{R}^{2}} d\xi \ |\xi|^{s+1}|\hat{f}(\xi)|  \leq   \int_{\mathbb{R}^{2}} d\xi \ |\xi|^{s}|\mathscr{F}(N(f))(\xi)|.$$
After several computations we can bound the right hand side of the inequality as
$$
 \int_{\mathbb{R}^{2}} d\xi \ |\xi|^{s}|\mathscr{F}(N(f))(\xi)| \lesssim \|f\|_{1},
$$ 
where the implicit constant depends on $s$, $k_{0}$ and $\|f_{0}\|_{H^{3}}$. We then use the time decay of $\|f\|_{1}(t)$ from the previous step as  $\|f\|_{1}(t) \lesssim (1+t)^{-1+\nu}$ for $-1 < \nu < s$, to obtain after integrating in time that $\|f\|_{s}(t)$ is indeed uniformly bounded in time. We further a uniform bound for the case $s = -1$ by an interpolation argument.

Lastly in the endpoint case $s=-2$ we prove bounds for the norm $\|f\|_{-2,\infty}$.  To accomplish this goal we prove uniform bounds on the integral over each annulus $C_{j}$. 

Once we have these uniform bounds, we use the general decay Lemma \ref{decaylemma} to obtain the decay result for $0\leq s \leq 1$:
$$\|f\|_{s}(t) \lesssim (1+t)^{-s+\nu} $$ where $-2\leq \nu < s$ and $\|f_{0}\|_{\nu} < \infty$ for $\nu > -2$ and $\|f_{0}\|_{-2,\infty} < \infty$ for $\nu = -2$.

Finally, to obtain time decay results for $1 < s < l - 1$, we utilize the decay of the norm $\|f\|_{1}(t)$. We control the time derivative of $\|f\|_{s}(t)$:
$$ 
\frac{d}{dt}\int_{\mathbb{R}^{2}}|\xi|^{s}|\hat{f}| \ d\xi \leq -\int_{\mathbb{R}^{2}} d\xi \ |\xi|^{s+1}|\hat{f}(\xi)| + \int_{\mathbb{R}^{2}} d\xi \ |\xi|^{s}|\mathscr{F}(N(f))(\xi)|.
$$ 
Next, for suitably large times we carefully control $\int_{\mathbb{R}^{2}} d\xi \ |\xi|^{s}|\mathscr{F}(N(f))(\xi)$  relative to the negative quantity $-\int_{\mathbb{R}^{2}} d\xi \ |\xi|^{s+1}|\hat{f}(\xi)| = -\|f\|_{s+1}$ by using the previously established time decay rates.  This enables us to establish an inequality of the form:
\bea\label{inequalityfordecay}
\frac{d}{dt}\|f\|_{s}(t) \leq -\delta \|f\|_{s+1}(t)
\eea
given $t \geq T$ for some $T > 0$. We indeed get the existence of such a time $T > 0$ by proving that
$$\int_{\mathbb{R}^{2}} d\xi \ |\xi|^{s}|\mathscr{F}(N(f))(\xi) \leq \pi\sum_{n\geq1}a_{n}(2n+1)^{s}\|f\|_{1}^{2n}\|f\|_{s+1}.$$ 
Then due to the large time decay of $\|f\|_{1}(t)$, there exists a time $T > 0$ such that (\ref{inequalityfordecay}) does indeed hold. By our uniform bound on $\|f\|_{-2,\infty}$ and using the decay Lemma \ref{decaylemma}, we obtain the large time decay results for $1< s < l-1$.

\subsection{Outline of the rest of the article}
We now outline the structure of the remainder of the article. In Section \ref{sec.embedding}, we prove embedding theorems to gain upper bounds on the s-norms of $f$ by appropriate Lebesgue and Sobolev norms. In Section \ref{Decayin3D}, we prove the main results in for the 3D Muskat problem. We first prove the general decay lemma for the $\|\cdot\|_{s}$ norms, which we then use to prove uniform bounds on $\|f\|_{s}(t)$ for $-1<s<2$. Next, we prove uniform bounds for the range $-2<s \leq 1$. Finally, we tackle the endpoint $s=2$ case by proving a uniform bound on the Besov-type norm $\|f\|_{-2,\infty}$, which allows us to prove decay results of the form
$$\|f\|_{s} \lesssim (1+t)^{-s+\nu} $$ for $1< s<l-1$ and $-2\leq \nu < s$ when the initial data $f_{0} \in H^{l}(\mathbb{R}^{2})$ satisfies the conditions outlined in Proposition \ref{higherdecay3D}. Finally, in Section \ref{2Dsection}, we conclude by outlining the analogous 2D results.

\section{Embeddings for $\|\cdot\|_{s}$ and $\|\cdot\|_{s,\infty}$}\label{sec.embedding}
In this section, we prove embeddings for the norms $\|\cdot\|_{s}$ and $\|\cdot\|_{s,\infty}$. We will later use these embeddings to gain uniform control of $\|f\|_{s}$ over a certain range of $s$ given by the embedding lemmas. We bound $\|\cdot\|_{s}$ from above by Sobolev norms because the well-posedness result of \cite{CCGRPS} is proven in a $L^{2}$-Sobolev space. We prove a more general embedding:

\begin{lemma}\label{lem.bounds.upper}
	For $s > -\frac{\Ddim}{p}$ and $r>s+\Ddim/q$ and $p,q \in [1,2]$ we have the inequality
	\begin{equation}\label{s.emb}
		\|f\|_{s} \lesssim \|f\|_{L^p(\mathbb{R}^d)}^{1-\theta}
		\|f\|_{\dot{W}^{r,q}(\mathbb{R}^\Ddim)}^{\theta},
	\end{equation}
	where $\theta = \frac{s+\Ddim / p}{r+\Ddim \left(\frac{1}{p}- \frac{1}{q}\right)}\in (0,1)$.

		For $s = -\frac{\Ddim}{p}$ and $p\in [1,2]$  we further have the inequality
	\begin{equation}\label{s.inf.emb}
		\|f\|_{s,\infty} \lesssim \|f\|_{L^p(\mathbb{R}^d)}.
	\end{equation}
	In particular for $s=-\Ddim$ we take $p=1$.
\end{lemma}

\begin{remark}
In particular for $s > -\frac{\Ddim}{2}$  then \eqref{s.emb} implies that 
\begin{equation} \label{hembed}
		\|f\|_{s} \lesssim \|f\|_{H^r(\mathbb{R}^d)} \quad (r > s +\Ddim/2).
\end{equation}

For exponents  $1\leq p \leq 2$, $r > s+\frac{d}{p}$ and $s > -\frac{d}{p}$, we also conclude
	\begin{equation}\notag
	  \|f\|_{s} \lesssim \|f\|_{W^{r,p}(\mathbb{R}^{\Ddim})}.
	\end{equation}
This follows directly from \eqref{s.emb}.

Also notice that  generally for $s \in (-\Ddim, - \Ddim/2]$ in \eqref{s.emb} we require $p \in [1, - \Ddim/s)$;  in particular this does not include $p=2$.  
\end{remark}

\begin{remark}
We very briefly introduce the Littlewood-Paley operators, $\triangle_{j}$ for $j\in\mathbb{Z}$ is defined on the Fourier side by
$$ 
\widehat{\triangle_{j}f} =\varphi_j \hat{f} = \varphi(2^{-j}\xi) \hat{f}(\xi),
$$ 
where $\varphi :\mathbb{R}^{\Ddim}\rightarrow [0,1]$ is a standard non-zero test function which is supported inside the annulus $\tilde{C}_{1} = \{3/4 \leq |\xi| \leq 8/3\}$ which contains the annulus $C_{1}$ (defined just below \eqref{normSinfty}).  The test function $\varphi$ is then normalized as 
$
\sum_{j\in \mathbb{Z}} \varphi(2^{-j} \xi) = 1 
$
$
\forall \xi \ne 0.
$
\end{remark}

\begin{proof} 
We use the Littlewood-Paley operators to obtain the estimate
\begin{equation}\notag
\int_{\mathbb{R}^\Ddim} |\xi|^{s} \varphi_j(\xi) |\hat{f}(\xi)| \ d\xi 
\approx  2^{js} \int_{\mathbb{R}^\Ddim}  \varphi_j(\xi)|\hat{f}(\xi)| \ d\xi.
\end{equation}
We then apply the Bernstein inequality followed by the Hausdorff-Young inequality to $\int_{\mathbb{R}^\Ddim}  \varphi_j(\xi) |\hat{f}(\xi)| \ d\xi$ to obtain for any $1 \le p \le 2$ and $\frac{1}{p}+\frac{1}{p'}=1$ that
\begin{equation}\label{cjupper}
\int_{\mathbb{R}^\Ddim} |\xi|^{s} \varphi_j(\xi) |\hat{f}(\xi)| \ d\xi 
\lesssim
  2^{js} 2^{j\frac{d}{p}} \|\widehat{\triangle_{j}f}\|_{L^{p'}(\mathbb{R}^d)}
  \lesssim
  2^{js} 2^{j\frac{d}{p}} \|\triangle_{j}f\|_{L^{p}(\mathbb{R}^d)}.
\end{equation}
Next we sum \eqref{cjupper} separately over $2^j \le R$ and $2^j > R$ for some $R>0$ to be chosen 
\begin{equation}\label{jlow}
\int_{|\xi| \le R} |\xi|^{s}|\hat{f}(\xi)| \ d\xi 
\lesssim  
\sum_{2^j \le R}2^{js} 2^{j\frac{d}{p}} \|\triangle_{j} f\|_{L^p(\mathbb{R}^\Ddim)}
	\lesssim
	R^{s+\Ddim/p}
	\| f\|_{L^p(\mathbb{R}^d)}.
\end{equation}
The last inequality holds for $s+\Ddim/p >0$.  

Now we sum \eqref{cjupper} over  $2^j > R$ and choose a possibly different $p=q$ to obtain
\begin{multline}\label{jhigh}
\int_{|\xi| > R} |\xi|^{s}  |\hat{f}(\xi)| \ d\xi 
\lesssim  
\sum_{2^j > R}2^{js} 2^{j\frac{\Ddim}{q}} \|\triangle_{j} f\|_{L^q(\mathbb{R}^\Ddim)}
\\
	\lesssim
	\left( \sum_{2^j > R}  2^{2 j\left( -r+s +\frac{\Ddim}{q}\right)} \right)^{1/2}
	\left( \sum_{2^j > R}  2^{2jr}\|\triangle_{j} f\|_{L^q(\mathbb{R}^\Ddim)}^2 \right)^{1/2}
			\\
	\lesssim
	R^{-r+s+\Ddim/q}
	\left\|\left( \sum_{2^j > R}  2^{2jr} \left|\triangle_{j} f\right|^2 \right)^{1/2} \right\|_{L^q(\mathbb{R}^\Ddim)}
			\\
	\lesssim
	R^{-r+s+\Ddim/q}
	\|f\|_{\dot{W}^{r,q}(\mathbb{R}^\Ddim)}.
\end{multline}
Here we used the Bernstein inequalities and the Minkowski inequality for norms since $1\le q \le 2$.  We further used the Littlewood-Paley characterization of $\dot{W}^{r,q}(\mathbb{R}^\Ddim)$.  
This inequality holds as soon as $r>s+\Ddim/q$.

Now to establish \eqref{s.emb}, in  \eqref{jlow} and \eqref{jhigh} we further choose
$$
R^{r+\Ddim \left(\frac{1}{p}-\frac{1}{q} \right)}
=
 \frac{\|f\|_{\dot{W}^{r,q}(\mathbb{R}^\Ddim)}}{ \| f\|_{L^p(\mathbb{R}^d)}},
$$
and then add the two inequalities together.

Lastly we show \eqref{s.inf.emb} by choosing $s = -\frac{\Ddim}{p}$ and $p\in [1,2]$ in \eqref{cjupper}.  
\end{proof}

We can also obtain lower bounds for the norms $\|f\|_{s}$ and $\|f\|_{s,\infty}$.  In particular 
\bea \label{embedlower}
	\|f\|_{\dot{W}^{s,\infty}(\mathbb{R}^\Ddim)} \lesssim \|f\|_{s}, 
\eea
which holds for any $s > -\Ddim$.	This inequality follows directly from the Hausdorff-Young inequality as
$
\|f\|_{\dot{W}^{s,\infty}} \leq  \||\xi|^{s}\hat{f}(\xi)\|_{L^{1}_{\xi}} = \|f\|_{s}.
$ 

We also have a lower bound given by the Besov norm:
$$ 
\|f\|_{\dot{B}^{s}_{\infty,\infty}} \eqdef \Big\|2^{js} \|\triangle_{j}f\|_{L^{\infty}}\Big\|_{l_j^{\infty}(\mathbb{Z})}
$$ 
For this norm, we have the following estimate by the Hausdorff-Young inequality:
\begin{equation}
  \|f\|_{\dot{B}^{s}_{\infty,\infty}} 
\lesssim  
\Big\|2^{js} \|\varphi(2^{-j}\xi) \hat{f}(\xi)\|_{L^{1}}\Big\|_{l_j^{\infty}} 
\lesssim \|f\|_{s,\infty}.
\end{equation}
And this holds for any $s \ge -\Ddim$ (including $s = -\Ddim$).

We point out here that one can interpolate between the $\|\cdot\|_{s}$ norms  as
\begin{equation} \label{Sinterpolate}
  \|f\|_{s} 
 \lesssim
 \|f\|_{\ss_1,\infty}^\theta \|f\|_{\ss_2,\infty}^{1-\theta}, 
 \quad 
 \ss_1 < s < \ss_2, \quad \theta = \frac{\ss_2-s}{\ss_2 - \ss_1}
\end{equation}
This inequality \eqref{Sinterpolate} can be seen in \cite[Lemma 4.2]{SohingerStrain}.  We will however give a short proof of \eqref{Sinterpolate} for completeness.  First notice that \eqref{Sinterpolate}  and \eqref{ineqbd} imply
\begin{equation} \notag 
  \|f\|_{s} 
 \lesssim
 \|f\|_{\ss_1}^\theta \|f\|_{\ss_2}^{1-\theta}, 
 \quad 
 \ss_1 \le s \le \ss_2, \quad \theta = \frac{\ss_2-s}{\ss_2 - \ss_1}.
\end{equation}
These inequalities show that if we have uniform control on for example $\|f\|_{1}$ and $\|f\|_{-2,\infty}$, then we also have uniform bounds on $\|f\|_{s}$ for $-2 < s \le 1$.

Now we prove \eqref{Sinterpolate}.  
For $R >0$ to be chosen later, using \eqref{cjupper} we expand out
$$
 \|f\|_{s}
\lesssim
 \sum_{j \in \mathbb{Z}}  \int_{\mathbb{R}^\Ddim} |\xi|^{s} \varphi_j(\xi) |\hat{f}(\xi)| \ d\xi 
 \lesssim
 \sum_{j \in \mathbb{Z}}  2^{js}  \|\triangle_{j}f\|_{L^1(\mathbb{R}^d)}
  \lesssim 
  \sum_{2^j \ge R} + \sum_{2^j < R}.
$$
For the first term
$$
\sum_{2^j \ge R} \lesssim \|f\|_{\ss_2, \infty}  \sum_{2^j \ge R} 2^{j(s-\ss_2)}  
\lesssim \|f\|_{\ss_2, \infty}   R^{s-\ss_2}  
$$
For the second term
$$
\sum_{2^j < R} \lesssim \|f\|_{\ss_1, \infty}  \sum_{2^j < R} 2^{j(s-\ss_1)}  
\lesssim \|f\|_{\ss_2, \infty}   R^{s-\ss_1}.  
$$
Then choose $R = \left( \|f\|_{\ss_2, \infty}/ \|f\|_{\ss_1, \infty} \right)^{1/(\ss_2-\ss_1)}$ to establish \eqref{Sinterpolate}.

Having established the relevant norm inequalities, we now move onto the proof of our main result.

\section{Decay in 3D}\label{Decayin3D}

We prove the decay results for the 3D Muskat problem in this section. First, we will establish a decay lemma, which will allow us to use the bounds we prove on $\|f\|_{s}$ and $\|f\|_{s,\infty}$ to obtain decay results for the interface. Next, we use the embedding theorems to get uniform bounds on $\|f\|_{s}$ for $-1<s <2$ and we use the decay lemma to get decay of the quantity $\|f\|_{1}(t)$. Finally, we use this decay to get new uniform bounds on $\|f\|_{s}$ for $-2<s\leq -1$ and $\|f\|_{-2,\infty}$. We conclude by using these new uniform bounds to prove faster time decay on $\|f\|_{s}(t)$ for the range of $0\leq s \leq  l-1$.

\subsection{The Decay Lemma}
In this section we now prove the  general decay lemma.  We will for now continue to work in $\mathbb{R}^\Ddim$ for an integer dimension $\Ddim \ge 1$.  In the next sub-sections we will use the following decay lemma to prove uniform bounds and decay in the $\|\cdot\|_{s}$ norm.  The following lemma proves a general time decay rate for solutions to the given differential inequality.

\begin{lemma}\label{decaylemma}
Suppose $g=g(t,x)$ is a smooth function with $g(0,x) = g_0(x)$ and assume that for some $\ss \in \mathbb{R}$, $\|g_0\|_{\ss} < \infty$ and 
$\|g(t)\|_{\rr,\infty} \leq C_{0}$ for some $\rr \ge -\Ddim$ satisfying  $ \rr < \ss$.
Let the following differential inequality hold for some $C>0$:
	$$
	\frac{d}{dt}\|g\|_{\ss} \leq -C \|g\|_{\ss +1}.
	$$ 
	  Then we have the uniform in time estimate
	$$ 
	\|g\|_{\ss}(t) \lesssim (1+t)^{-\ss+\rr}.
	$$
\end{lemma}

\begin{remark}\label{remarkdecay}
Note that by \eqref{ineqbd} we have $\|f\|_{\rr,\infty}(t) \leq \|f\|_{\rr}(t)$. Therefore we can use Lemma \ref{decaylemma} if we can bound $\|f\|_{\rr}(t)$ for $\rr > -\Ddim$ uniformly in time.
\end{remark}

\begin{proof}
	For some $\delta, \kappa>0$ to be chosen, we initially observe that
	\begin{align*}
	\|g\|_{\kappa} &=\int_{\mathbb{R}^{d}}|\xi|^{\kappa}|\hat{g}(\xi)|d\xi \\
	&\geq \int_{|\xi|>(1+\delta t)^{s}}|\xi|^{\kappa}|\hat{g}(\xi)|d\xi \\
	&\geq (1+\delta t)^{s\beta}\int_{|\xi|>(1+\delta t)^{s}}|\xi|^{\kappa-\beta}|\hat{g}(\xi)|d\xi\\
	&= (1+\delta t)^{s\beta}\Big(\|g\|_{\kappa-\beta} \  - \int_{|\xi|\leq(1+\delta t)^{s}}|\xi|^{\kappa-\beta}|\hat{g}(\xi)|d\xi\Big)
	\end{align*}
	Using this inequality with $\kappa = \ss + 1$ and $\beta = 1$, we obtain that
\begin{multline*}
\frac{d}{dt}\|g\|_{\ss} + C(1+\delta t)^{s}\|g\|_{\ss} \leq -C\|g\|_{\ss+1} + C(1+\delta t)^{s}\|g\|_{\ss} 
\\
\leq C(1+\delta t)^{s}\int_{|\xi| \leq(1+\delta t)^{s}}|\xi|^{\ss}|\hat{g}(\xi)|d\xi.
\end{multline*}	
	Then, using the sets $C_j$ as in \eqref{normSinfty} and defining $\chi_{S}$ to be the characteristic function on a set $S$, the upper bound in the last inequality can be bounded as follows
	\begin{align*}
	\int_{|\xi|\leq(1+\delta t)^{s}}|\xi|^{\ss}|\hat{g}(\xi)|d\xi
	&= \sum_{j\in\mathbb{Z}} \int_{C_{j}}\chi_{\{|\xi|\leq (1+\delta t)^{s}\}}|\xi|^{\ss}|\hat{g}| \ d\xi\\ 
	&\approx \sum_{2^{j}\leq (1+\delta t)^{s}} \int_{C_{j}} |\xi|^{\ss}|\hat{g}| \ d\xi\\
	&\lesssim \|g\|_{\rr,\infty}\sum_{2^{j}\leq (1+\delta t)^{s}} 2^{j(\ss-\rr)}\\
	&\lesssim \|g\|_{\rr,\infty} (1+\delta t)^{s(\ss-\rr)}\sum_{2^{j}(1+\delta t)^{-s}\leq 1} 2^{j(\ss-\rr)}(1+\delta t)^{-s(\ss-\rr)} \\
	&\lesssim \|g\|_{\rr,\infty}(1+\delta t)^{s(\ss-\rr)}
	\end{align*}
	where the implicit constant in the inequalities do not depend on $t$.  In particular we have used that the following uniform in time estimate holds
	$$
	\sum_{2^{j}(1+\delta t)^{-s}\leq 1} 2^{j(\ss-\rr)}(1+\delta t)^{-s(\ss-\rr)} \lesssim 1.
	$$
	Combining the above inequalities, we obtain that
	\bea\label{use}
	\frac{d}{dt}\|g\|_{\ss} + C(1+\delta t)^{s}\|g\|_{\ss} \lesssim C_0 (1+\delta t)^{s}(1+\delta t)^{s(\ss-\rr)}.
	\eea
In the following estimate will use \eqref{use} with $s=-1$, we suppose $a > \ss-\rr >0$, and we choose $\delta > 0$ such that $a\delta = C$.  We then obtain that
	\begin{align*}
	\frac{d}{dt}((1+\delta t)^{a}\|g\|_{\ss}) &=  (1+\delta t)^{a}\frac{d}{dt}\|g\|_{\ss} + a\delta \|g\|_{\ss}(1+\delta t)^{a-1} \\
	&=  (1+\delta t)^{a}\frac{d}{dt}\|g\|_{\ss} + C \|g\|_{\ss}(1+\delta t)^{a-1} \\
	&\leq (1+\delta t)^{a}\Big(\frac{d}{dt}\|g\|_{\ss} + C(1+\delta t)^{-1}\|g\|_{\ss}\Big) \\
	&\lesssim C_0 (1+\delta t)^{a-1-(\ss-\rr)}
	\end{align*}
	Since $a > \ss-\rr$, we integrate in time to obtain that
	$$ 
	(1+\delta t)^{a}\|g\|_{\ss} \lesssim \frac{C_0}{\delta} (1+\delta t)^{a-(\ss-\rr)}.
	$$ We conclude our proof by dividing both sides of the inequality by $(1+\delta t)^{a}$.
\end{proof}

Lemma \ref{decaylemma} shows that to prove the time decay rates claimed in Theorem \ref{mainthm} then it is sufficient to establish suitable differential inequalities and also to prove uniform in time bounds  on the norms $\| \cdot \|_{s}$.    

Starting now we will switch our focus to only talking about the 3D case (with $\Ddim =2$) in \eqref{interfaceq}.  Looking at establishing the differential inequality first, from \cite{JEMS,CCGRPS} we have the differential inequality \eqref{1diffbound} for $\| \cdot \|_{1}$.  Furthermore, from \cite{CCGRPS}, we also know that for $0<\delta < 1$ and $k_{0}$ satisfying (\ref{k0}) that
\bea \notag
\|f\|_{1+\delta}(t) + \mu \int_{0}^{t}\ ds \ \|f\|_{2+\delta}(s) \leq \|f_{0}\|_{1+\delta}
\eea
where $\mu>0$ depends on $\|f_{0}\|_{1}$.  It is also shown in \cite{CCGRPS} that
\bea\label{sobolevuniform}
\|f\|_{H^{l}(\mathbb{R}^2)}(t) \leq \|f_{0}\|_{H^{l}(\mathbb{R}^2)}\exp(CP(k_{0})\|f_{0}\|_{1+\delta}/\mu),
\quad l \ge 3,
\eea
where $C>0$ is a constant and $P(k_{0})$ is a polynomial in $k_{0}$.   Furthermore following the exact proof of \eqref{1diffbound} in \cite{JEMS,CCGRPS} one can directly observe that
\bea\label{1Sdiffbound}
\frac{d}{dt}\|f\|_{s}(t) \leq -C \|f\|_{s+1},  \quad 0 \le s \le 1.
\eea
We will use this differential inequality in the following to prove the time decay rates in Theorem \ref{mainthm}.  Later in the proof of Proposition \ref{higherdecay3D} we will establish \eqref{1Sdiffbound} for $1 \le s \le l - 1$. First, we use (\ref{sobolevuniform}), (\ref{hembed}) and (\ref{1Sdiffbound}) to obtain uniform bounds on $\|f\|_{s}(t)$ in the range $-1<s < 2$ and an initial decay result for $\|f\|_{s}(t)$ in the range $0\leq s \leq 1$.

\subsection{Uniform Bounds for $-1<s<2$}

In this subsection we will establish uniform in time bounds for $\|f\|_{s}(t)$ when $-1<s<2$ and then we use those to prove an time decay for $\|f\|_{s}(t)$ when $s \in [0, 1]$.

Lemma \ref{lem.bounds.upper}, in particular \eqref{hembed}, immediately grants the following corollary.
\begin{cor}\label{initialboundcor}
Suppose $f$ is the solution to the Muskat problem \eqref{interfaceq} in 3D  described by Theorem \ref{oldthm3D}.
Then, for $-1<s<2$, we have the uniform in time estimate
$$
\|f\|_{s}(t)\lesssim 1.
$$
Here the implicit constant depends upon $\|f_0\|_{H^3}$.
\end{cor}

\begin{proof}
By (\ref{sobolevuniform}), we know that $\|f\|_{H^{3}}(t)$ is uniformly bounded in time since from 
\eqref{s.emb} we have $\|f_0\|_{1+\delta} \lesssim \|f_0\|_{H^{3}}$. 
Further directly from \eqref{s.emb} we have that
$$
\|f\|_{s}(t) \lesssim \|f\|_{H^{3}}(t) \lesssim 1
$$ 
holds uniformly in time for $-1 < s < 2$.
\end{proof}

We now apply the decay Lemma \ref{decaylemma} to the special case $\ss =s \in [0, 1]$, then using also Corollary \ref{initialboundcor} we obtain 

\begin{prop}
Suppose $f$ is the solution to the Muskat problem \eqref{interfaceq} in 3D  described by Theorem \ref{oldthm3D}.
	Then for $s \in [0, 1]$ we have the uniform in time estimate
	\bea\label{initialdecay}
	\|f\|_{s} \lesssim (1+t)^{-s+\rr},
	\eea
	for any $-1 < \rr < s$; above the implicit constant depends on $\|f_{0}\|_{\rr}$.  
\end{prop}

Having established some decay of the interface, we will now be able to use the decay for the specific case $s=1$ to prove uniform bounds for $-2<s\leq -1$.

\subsection{Uniform Bounds for $-2<s\leq -1$}	
For the 3D Muskat problem \eqref{interfaceq}, 
we will use the time decay estimate (\ref{initialdecay}) to prove uniform in time bounds for $\|f\|_{s}$ for $-2 < s \leq -1$. First, we establish the following estimate

\begin{prop}\label{prop9}
Suppose $f$ is the solution to the Muskat problem \eqref{interfaceq} in 3D  described by Theorem \ref{oldthm3D} with $\|f_{0}\|_{s}< \infty$ for some $-2<s<-1$.   Then 
\bea\label{decay1}
\frac{d}{dt} \|f \|_{s} (t) \lesssim \|f\|_{1},
\eea
where the implicit constant depends on $s$, $k_{0}$ and $\|f_{0}\|_{H^{3}}$.
\end{prop}
	
\begin{proof}
Following the computation of the time derivative of $\|f\|_{1}(t)$ in the proof of Theorem 3.1 in \cite{JEMS}, we can prove that
\begin{equation}
	\label{energyINEQ}
\frac{d}{dt} \|f\|_{s} (t) +C\int_{\mathbb{R}^{2}} d\xi \ |\xi|^{s+1}|\hat{f}(\xi)|  \leq   \int_{\mathbb{R}^{2}} d\xi \ |\xi|^{s}|\mathscr{F}(N(f))(\xi)|.
\end{equation}
We can bound $|\mathscr{F}(N(f))(\xi)|$ as in \cite{CCGRPS} to get the bound: 
\begin{multline}\label{firstineq}
\int_{\mathbb{R}^{2}} d\xi \ |\xi|^{s}|\mathscr{F}(N(f))(\xi)| 
\\
\leq \pi \sum_{n\geq 1} a_{n} \int_{\mathbb{R}^{2}} \int_{\mathbb{R}^{2}} \cdots \int_{\mathbb{R}^{2}} |\xi|^{s} |\xi - \xi_{1}| |\hat{f}(\xi-\xi_{1})| 
\\
\times
\prod_{j=1}^{2n-1}|\xi_{j} - \xi_{j+1}| |\hat{f}(\xi_{j} - \xi_{j+1})| |\xi_{2n}||\hat{f}(\xi_{2n})| d\xi d\xi_{1} \ldots d\xi_{2n}
\end{multline}
where $$a_{n} = \frac{(2n+1)!}{(2^{n}n!)^{2}}.$$ 
Given a function $g$, define the corresponding function $\tilde{g}$ by $\tilde{g}(x) = g(-x)$. Then, since $|x-y| = |y-x|$ for any $x,y \in\mathbb{R}^{2}$, we obtain:
\begin{multline}
\int_{\mathbb{R}^{2}} d\xi \ |\xi|^{s}|\mathscr{F}(N(f))(\xi)| 
\\
\leq \pi \sum_{n\geq 1} a_{n} \int_{\mathbb{R}^{2}} \int_{\mathbb{R}^{2}} \cdots \int_{\mathbb{R}^{2}} |\xi|^{s} |\xi_{1} - \xi| |\tilde{\hat{f}}(\xi_{1}-\xi)| 
\\
\times
\prod_{j=1}^{2n-1}|\xi_{j+1} - \xi_{j}| |\tilde{\hat{f}}(\xi_{j+1} - \xi_{j})| |\xi_{2n}||\hat{f}(\xi_{2n})| d\xi d\xi_{1} \ldots d\xi_{2n}.
\end{multline}
Hence, writing the right hand side in terms of convolutions, we obtain that		
\begin{multline*}
	\int_{\mathbb{R}^{2}} d\xi \ |\xi|^{s}|\mathscr{F}(N(f))(\xi)| 
	\\
	\leq \pi \sum_{n\geq 1} a_{n}  \int_{\mathbb{R}^{2}}  |\xi_{2n}||\hat{f}(\xi_{2n})| \Big(|\cdot|^{s} \ast \Big(\ast^{2n}|\cdot||\tilde{\hat{f}}(\cdot)| \Big) \Big)(\xi_{2n}) d\xi_{2n} 
\end{multline*}
Applying Holder's inequality:
\begin{multline}
	\label{holder}
\int_{\mathbb{R}^{2}}  |\xi_{2n}||\hat{f}(\xi_{2n})| \Big(|\cdot|^{s} \ast (\ast^{2n} |\cdot||\tilde{\hat{f}}(\cdot)|)  \Big)(\xi_{2n}) d\xi_{2n} 
\\
\leq \||\cdot||\hat{f}(\cdot)|\|_{L^{1}} \||\cdot|^{s} \ast ( \ast^{2n} |\cdot||\tilde{\hat{f}}(\cdot)|)  \|_{L^{\infty}}
\end{multline}
The first term on the right hand side of (\ref{holder}) is exactly $\|f\|_{1}$ which is bounded. The second term can be controlled first by Young's inequality with $\frac{1}{p}+\frac{1}{q} = 1$:
\bea\label{young}
\||\cdot|^{s} \ast ( \ast^{2n} |\cdot||\tilde{\hat{f}}(\cdot)|)  \|_{L^{\infty}}
\leq \||\cdot|^{s} \ast ( \ast^{2n-1} |\cdot||\tilde{\hat{f}}(\cdot)|)  \|_{L^{q}}\||\cdot||\tilde{\hat{f}}(\cdot)| \|_{L^{p}}
\eea
where we choose $q \in (2,\infty)$ such that $\frac{1}{q} = \frac{-s-1}{2}$. Thus, $p\in (1,2)$, so we use interpolation to obtain  for $\frac{\theta}{1} + \frac{1-\theta}{2} = \frac{1}{p}$ that
\bea\label{interpolation}
\||\cdot||\tilde{\hat{f}}(\cdot)| \|_{L^{p}} = \||\cdot||\hat{f}(\cdot)| \|_{L^{p}} \leq \||\cdot||\hat{f}(\cdot)| \|_{L^{1}}^{\theta}\||\cdot||\hat{f}(\cdot)| \|_{L^{2}}^{1-\theta} \leq \|f\|_{1}^{\theta}\|f\|_{H^{3}}^{1-\theta}.
\eea
We control the other term by the Hardy-Littlewood-Sobolev inequality since $q\in (2,\infty)$:	
\bea\label{hls}
\||\cdot|^{s} \ast ( \ast^{2n-1} |\cdot||\tilde{\hat{f}}(\cdot)|)  \|_{L^{q}} \lesssim \| \ast^{2n-1} |\cdot||\tilde{\hat{f}}(\cdot)| \|_{L^{2}}
\eea
since $-2< s < -1$ and our choice of $q$ enables the equality $1+ \frac{1}{q} = -\frac{s}{2} + \frac{1}{2}$. Finally we use Young's inequality with $1+\frac{1}{2} = 1 + \frac{1}{2}$ repeatedly to control the $2n-1$ convolutions and get the bound
\bea\label{repeatedyoung}
\|\ast^{2n-1} |\cdot||\tilde{\hat{f}}(\cdot)|  \|_{L^{2}} \leq \||\cdot||\tilde{\hat{f}}(\cdot)| \|_{L^{2}}\||\cdot||\tilde{\hat{f}}(\cdot)| \|_{L^{1}}^{2n-2} \leq \|f \|_{H^{3}}\|f \|_{1}^{2n-2},
\eea
where we have used the inequality:
$$ \||\xi||\hat{f}|\|_{L^{2}_{\xi}} \leq \| (1+|\xi|^{2})^{\frac{3}{2}}|\hat{f}|\|_{L^{2}_{\xi}} = \|f\|_{H^{3}}$$
Combining the above estimates from (\ref{holder}), (\ref{young}), (\ref{interpolation}), (\ref{hls}) and (\ref{repeatedyoung}), we obtain the following bound
$$
	\notag
\int_{\mathbb{R}^{2}}  |\xi_{2n}||\hat{f}(\xi_{2n})| \Big(|\cdot|^{s} \ast (\ast^{2n} |\cdot||\tilde{\hat{f}}(\cdot)|)  \Big)(\xi_{2n}) d\xi_{2n} 
\lesssim
\|f\|_{H^{3}}^{2-\theta}\|f\|_{1}^{2n-1+\theta}.
$$
Summing over all $n$, we get from (\ref{firstineq}) that
\begin{multline}
\label{finalineq}
\int_{\mathbb{R}^{2}} d\xi \ |\xi|^{s}|\mathscr{F}(N(f))(\xi)| 
\\
\lesssim  
\|f\|_{H^{3}}^{2-\theta}\|f\|_{1}^{\theta}\sum_{n\geq 1} a_{n}\|f\|_{1}^{2n-1} 
\lesssim
 \|f\|_{H^{3}}^{2-\theta}\|f\|_{1}^{\theta}\sum_{n\geq 0} a_{n+1}\|f\|_{1}^{2n+1}.
\end{multline}
By Theorem 3.1 in \cite{CCGRPS}, $\|f\|_{1} \leq \|f_{0}\|_{1} < k_{0}$. Further, given this bound on $\|f\|_{1}$, the above series converges.  Then by \eqref{sobolevuniform} we also know that 
$
\|f\|_{H^{3}(\mathbb{R}^2)}(t) \lesssim 1
$
uniformly in time.  Hence the following uniform bounds hold independently of $n$
\begin{multline} \notag
	 \int_{\mathbb{R}^{2}} d\xi \ |\xi|^{s}|\mathscr{F}(N(f))(\xi)| 
\lesssim
\sum_{n\geq 0} a_{n+1}\|f\|_{1}^{2n+1} 
\lesssim
\|f\|_{1}\sum_{n\geq 0} a_{n+1}\|f\|_{1}^{2n} 
\\
\lesssim
\|f\|_{1}\sum_{n\geq 0} a_{n+1}k_{0}^{2n} 
\lesssim
\|f\|_{1}.
\end{multline}
Here the uniform constant depends on $s$ and $k_{0}$. 
\end{proof}

Combining Proposition \ref{prop9} and \eqref{initialdecay} we obtain for example
$$ 
\frac{d}{dt} \|f \|_{s} (t) \lesssim (1+t)^{-1-\epsilon},
$$
for a small $\epsilon >0$. 
Then we integrate this to obtain
$$ 
\|f\|_{s}(t) \lesssim \|f_{0}\|_{s} + 1 + (1+t)^{-\epsilon}
$$ 
We conclude that $\|f\|_{s} \lesssim 1$ uniformly in time for $-2<s<-1$.  

In order to obtain the uniform bound for $s=-1$ we observe that
\begin{align}\label{interpbd}
\|f\|_{-1} &= \int_{\mathbb{R}^{2}} |\xi|^{-1} |\hat{f}(\xi)| \ d\xi 
\\
&= \int_{|\xi|\leq 1} |\xi|^{-1} |\hat{f}(\xi)| \ d\xi + \int_{|\xi|> 1} |\xi|^{-1} |\hat{f}(\xi)| \ d\xi 
\notag
\\
&\leq \int_{|\xi|\leq 1} |\xi|^{-2+\gamma} |\hat{f}(\xi)| \ d\xi + \int_{|\xi|> 1} |\xi| |\hat{f}(\xi)| \ d\xi 
\notag
\\
&\leq \|f\|_{-2+\gamma} + \|f\|_{1} \lesssim 1,
\notag
\end{align}
where $0<\gamma<1$. Hence, we have uniform in time bounds for $\|f\|_{s}$ for any $-2<s\leq -1$. We can now use Lemma \ref{decaylemma} to conclude the time decay
\bea\label{decayepsilon}
\|f\|_{\ss}(t) \lesssim (1+t)^{-\ss+\rr},
\eea
which holds for any $\ss \in [0, 1]$ and any $\rr \in (-2, \ss)$ where the implicit constant in particular depends on $\|f\|_{\rr}$ and $k_{0}$.
 We summarize this in the following proposition.

\begin{prop}\label{prop10}
Suppose $f$ is the solution to the Muskat problem in 3D described in Theorem \ref{oldthm3D}. Then we have uniformly for $-2<s\leq -1$ the following estimate
$$
\|f\|_{s}(t) \lesssim 1,
$$  
where the implicit constant depends on $k_{0}$ and $\|f_0\|_{s}<\infty$.  And the decay estimate (\ref{decayepsilon}) holds.
\end{prop}

This establishes uniform bounds for a larger range of $s$. We will now prove bounds on the endpoint case.

\subsection{The Endpoint Case $s=-2$} 
To prove the uniform bounds for the endpoint case $s=-2$, we use the Besov-type norm from \eqref{normSinfty} which we recall as
\bea \notag
\|f\|_{-2,\infty} = \Big\|\int_{C_{j}}|\xi|^{-2}|\hat{f}(\xi)|d\xi \Big\|_{l^{\infty}_{j}},
\eea
where we further recall the annulus $C_{j} = \{2^{j-1} \leq |\xi| < 2^{j}\}$. 

\begin{prop}\label{endpointprop}
 Let $f$ be the unique solution to the Muskat problem in 3D from Theorem \ref{oldthm3D}. Then the following estimate holds uniformly in time
\begin{equation}
\label{boundendpoint}
\|f\|_{-2,\infty}(t)\lesssim 1,
\end{equation}
where the implicit constant depends on  $\|f_{0}\|_{-2,\infty} < \infty$ and $k_{0}$.
\end{prop}

\begin{proof}
To control this endpoint norm, we uniformly bound the integral over $C_{j}$ for each $j\in\mathbb{Z}$. Analogous to the proof of \eqref{energyINEQ} from \cite[Theorem 3.1]{JEMS}, we can use the same exact argument to show that
\begin{multline}\label{energyCj}
	\frac{d}{dt} \int_{C_{j}}|\xi|^{-2}|\hat{f}(\xi)|d\xi + C\int_{C_{j}} d\xi \ |\xi|^{-1}|\hat{f}(\xi)| 
\leq  \int_{C_{j}} d\xi \ |\xi|^{-2}|\mathscr{F}(N(f))(\xi)|
\end{multline}
Note that on the annulus $C_{j}$ the term $|\xi|^{-2}$ is bounded above and below.  
Next, we control the term on the right hand side using analogous estimates on the integrand as we did for \eqref{energyINEQ}, the difference is that now we control $|\xi|^{-2}$ by the inner radius of the annulus:
\begin{multline*}
\int_{C_{j}} d\xi \ |\xi|^{-2}|\mathscr{F}(N(f))(\xi)| 
\\ 
\leq 2^{-2j+2}\pi \sum_{n\geq 1} a_{n} 
\int_{C_{j}} d\xi \int_{\mathbb{R}^{2}} d\xi_{1} \cdots \int_{\mathbb{R}^{2}} d\xi_{2n} ~ |\xi - \xi_{1}| |\hat{f}(\xi-\xi_{1})| 
\\
\times
\prod_{j=1}^{2n-1}|\xi_{j} - \xi_{j+1}| |\hat{f}(\xi_{j} - \xi_{j+1})| |\xi_{2n}||\hat{f}(\xi_{2n})|.  
\end{multline*}
Writing this integral in terms of convolutions, we obtain:
\begin{equation}\notag
	\int_{C_{j}} d\xi \ |\xi|^{-2}|\mathscr{F}(N(f))(\xi)|  \leq 2^{-2j+2}\pi \sum_{n\geq 1} a_{n} \int_{C_{j}}\Big(\ast^{2n+1} |\cdot||\hat{f}(\cdot)|\Big)(\xi) \ d\xi.
\end{equation}
Next, we obtain:
\bea \notag
\int_{C_{j}} d\xi \ |\xi|^{-2}|\mathscr{F}(N(f))(\xi)|  \leq 4\pi \sum_{n\geq 1} a_{n} \|\ast^{2n+1} |\cdot||\hat{f}(\cdot)|\|_{L^{\infty}},
\eea
since the size of the annulus $C_{j}$ can be bounded above by $2^{2j}$. Next by using Young's inequality, first with $1 + \frac{1}{\infty} = \frac{1}{2} + \frac{1}{2}$ and then with $1+\frac{1}{2}=1+\frac{1}{2}$, we obtain:
\begin{align*}
\int_{C_{j}} d\xi \ |\xi|^{-2}|\mathscr{F}(N(f))(\xi)|  &\leq 4\pi \sum_{n\geq 1} a_{n} \||\cdot||\hat{f}(\cdot)|\|^{2}_{L^{2}}\||\cdot||\hat{f}(\cdot)|\|_{L^{1}}^{2n-1}\\
&\leq  4\pi \|f\|^{2}_{H^{3}} \sum_{n\geq 1} a_{n}\|f\|_{1}^{2n-1} \\ &\leq  4\pi \|f\|^{2}_{H^{3}} \sum_{n\geq 0} a_{n+1}\|f\|_{1}^{2n+1}
\end{align*}
Since $\|f\|_{1} \leq k_{0}$, we obtain that:

$$\int_{C_{j}} d\xi \ |\xi|^{-2}|\mathscr{F}(N(f))(\xi)| \leq 4\pi \|f\|^{2}_{H^{3}}\|f\|_{1} \sum_{n\geq 0} a_{n+1}k_{0}^{2n}.$$
By the uniform bound on $\|f\|_{H^{3}}$ and since the series $\sum_{n\geq 0} a_{n+1}k_{0}^{2n}$ converges, we conclude that we have the following uniform in $j$ estimate
$$
\int_{C_{j}} d\xi \ |\xi|^{-2}|\mathscr{F}(N(f))(\xi)| \lesssim \|f\|_{1}.
$$ 
Finally, since for example $\|f\|_{1} \lesssim (1+t)^{-\frac{3}{2}}$ by \eqref{initialdecay}, we see that 
$$
\int_{C_{j}} d\xi \ |\xi|^{-2}|\hat{f}(\xi)| \lesssim  (1+t)^{-\frac{3}{2}}
$$ 
for a uniform constant which is independent of $j$. We then integrate \eqref{energyCj} in time to conclude that we have the uniform in time bound \eqref{boundendpoint}.
\end{proof}

The uniform bound on this endpoint case allows us to prove stronger decay of $\|f\|_{s}(t)$ for $0\leq s \leq 1$ by using Lemma \ref{decaylemma}. We will also now prove decay for the case $1 <  s <  l-1$ using the decay of the norm $\|f\|_{1}(t)$.

\subsection{General Decay Estimates}

Finally, we obtain the main time decay estimates for the Muskat equation in 3D. Using (\ref{1diffbound}) and (\ref{boundendpoint}), we can apply Lemma \ref{decaylemma} to obtain

\begin{cor}
For the solution $f$ to the Muskat problem in 3D described in Theorem 1, we have the following uniform in time decay estimate:
\bea\label{thmproved}
\|f\|_{s}(t)\lesssim (1+t)^{-s+\rr},
\eea
where we allow $s$ to satisfy $0 \le s \le 1$ and we allow $\rr$ to satisfy $-2 \le \rr < s$.  

When $\rr > -2$ then we require 
additionally that $\|f_{0}\|_{\rr} < \infty$, and when $\rr = -2$ then we alternatively require $\|f_{0}\|_{-2,\infty} < \infty$.  The implicit constant in \eqref{thmproved} depends on  either  $\|f_{0}\|_{\rr}$ (when $\rr > -2$) or $\|f_{0}\|_{-2,\infty}$ (when $\rr = -2$), $\|f_0\|_{s}$ and $k_{0}$.
\end{cor}

This corollary is Theorem \ref{mainthm} in 3D for $0 \le s \le 1$.  To establish Theorem \ref{mainthm} in 3D in the case $s >1$, we further make the following observation:

\begin{prop}\label{higherdecay3D}
Suppose $s$ satisfies $1< s < l-1$ and $f_{0}\in H^{l}(\mathbb{R}^{2})$ for some $l \ge 3$. If $\|f_{0}\|_{1} < k_{0}$ and $\|f_{0}\|_{s}< \infty$, then for the solution $f$ described in Theorem 1, we have the decay estimates
\bea\label{higherdecay}
\|f\|_{s}(t)\lesssim (1+t)^{-s+\rr}
\eea
where we allow $\rr$ to satisfy $-2 \le \rr < s$.  

When $\rr > -2$ then we require 
additionally that $\|f_{0}\|_{\rr} < \infty$, and when $\rr = -2$ then we alternatively require $\|f_{0}\|_{-2,\infty} < \infty$.  The implicit constant in \eqref{higherdecay} depends on  either  $\|f_{0}\|_{\rr}$ (when $\rr > -2$) or $\|f_{0}\|_{-2,\infty}$ (when $\rr = -2$), $\|f_0\|_{s}$ and $k_{0}$.
\end{prop}

\begin{proof}
First, as in \eqref{energyINEQ}, we have the following inequality
\bea\label{energ}
\frac{d}{dt}\int_{\mathbb{R}^{2}}|\xi|^{s}|\hat{f}| \ d\xi \leq -\int_{\mathbb{R}^{2}} d\xi \ |\xi|^{s+1}|\hat{f}(\xi)| + \int_{\mathbb{R}^{2}} d\xi \ |\xi|^{s}|\mathscr{F}(N(f))(\xi)|.
\eea
Next, following the arguments of \cite{CCGRPS} and \cite{JEMS}, 
we directly obtain that
\begin{multline}\label{higherfirstineq}
\int_{\mathbb{R}^{2}} d\xi \ |\xi|^{s}|\mathscr{F}(N(f))(\xi)| 
\\ 
\leq \pi\sum_{n\geq}a_{n}\int_{\mathbb{R}^{2}}d\xi\int_{\mathbb{R}^{2}}d\xi_{1}\cdots \int_{\mathbb{R}^{2}}d\xi_{2n} |\xi|^{s}|\xi-\xi_{1}| |\hat{f}(\xi-\xi_{1})| 
\\
\times
\prod_{j=1}^{2n-1}|\xi_{j} - \xi_{j+1}| |\hat{f}(\xi_{j} - \xi_{j+1})| |\xi_{2n}||\hat{f}(\xi_{2n})|.
\end{multline}
We use the inequality for $s>1$
\bea\label{trianglepower}
|\xi|^{s} \leq (2n+1)^{s-1}(|\xi-\xi_{1}|^{s}+|\xi_{1}-\xi_{2}|^{s}+\ldots+|\xi_{2n-1}-\xi_{2n}|^{s}+|\xi_{2n}|^{s}).
\eea
Applying (\ref{trianglepower}) and Young's inequality to (\ref{higherfirstineq}), it can be shown that
\bea\label{highersum}
\int_{\mathbb{R}^{2}} d\xi \ |\xi|^{s}|\mathscr{F}(N(f))(\xi)| \leq \pi\sum_{n\geq1}a_{n}(2n+1)^{s}\|f\|_{1}^{2n}\int_{\mathbb{R}^{2}}d\xi \ |\xi|^{s+1}|\hat{f}(\xi)|.
\eea
Hence, by (\ref{energ}) have that
\bea
\frac{d}{dt}\|f\|_{s}(t) \leq -\delta(t) \|f\|_{s+1}(t)
\eea
where $$\delta(t) = 1-\pi\sum_{n\geq1}a_{n}(2n+1)^{s}\|f\|_{1}(t)^{2n}.$$ By Theorem \ref{mainthm}, we know that if $\|f_{0}\|_{1} < k_{0}$, then
(\ref{thmproved}) holds. 

Thus, there exists some $T>0$ such that $\|f\|_{1}(T)$ is small enough such that $\delta(T) > \delta > 0$ for some constant $\delta > 0$. Since $\|f\|_{1}(t) \leq \|f\|_{1}(T)$ for $t\geq T$, we know that $\delta(t) > \delta(T) > \delta > 0$. Thus,

\bea\label{decreasing}
\frac{d}{dt}\|f\|_{s}(t) \leq -\delta \|f\|_{s+1}(t)
\eea
for all $t \geq T$. Now, consider the interface function $f_{T}$ defined by $f_{T} = f(t+T)$ defined for $t\geq T$. Then, $f_{T}$ satisfies the interface equation (\ref{interfaceq}) with initial condition $f_{T}(0) = f(T)$. For the case $\nu > -2$, since $\|f_{0}\|_{\nu} < \infty$, we know by Corollary \ref{initialboundcor} and Proposition \ref{prop10} that $\|f_{T}(0)\|_{\nu} = \|f\|_{\nu}(T) <\infty$ and $\|f_{T}\|_{\nu} \lesssim 1$ uniformly in time. For the case $\nu = 2$, since $\|f_{0}\|_{-2,\infty} < \infty$, we know by Proposition \ref{endpointprop} that $\|f_{T}(0)\|_{-2,\infty} =  \|f\|_{-2,\infty}(T) < \infty$ and $\|f_{T}\|_{-2,\infty} \lesssim 1$ uniformly in time. Further, by (\ref{decreasing}), 
$$ \frac{d}{dt}\|f_{T}\|_{s}(t) \leq -\delta \|f_{T}\|_{s+1}(t).$$
Hence, we can apply Lemma \ref{decaylemma} to $f_{T}$ to obtain the decay:
\bea\label{largetimedecay}
\|f_{T}\|_{s}(t) \leq \gamma (1+t)^{-s+\nu}.
\eea
Since $f(t) = f_{T}(t-T)$, we have the following decay estimate for $t\geq T$,
$$
\|f\|_{s}(t) \leq \gamma (1+t-T)^{-s+\nu} \leq \gamma (1+T)^{s+\nu}(1+t)^{-s+\nu}. 
$$  
Further, for $0\leq t\leq T$: $$\|f\|_{s}(t) \leq \|f\|_{H^{l}}(t) \leq C_{l}$$ 
where $C_{l}=\|f_{0}\|_{H^{l}}\exp(CP(k_{0})\|f_{0}\|_{1+\delta}/\mu)$ is the constant given by (\ref{sobolevuniform}). Collecting these last few estimates establishes the result.
\end{proof}

We have now established the decay results for the 3D Muskat problem. Similar results can be summarized for the 2D problem as well.

\section{Decay in 2D}\label{2Dsection}

In this last section, we will sketch the proof of the large time decay results for the 2D Muskat problem \eqref{interfaceq2D} given in Theorem \ref{mainthm} when $\Ddim=1$.   The proof is analogous to the 3D case that was just shown.

To prove the decay, we will first establish the uniform bounds of the relevant norms.  Firstly from \eqref{hembed} we obtain the uniform in time bound
\begin{equation} \label{sbound2Dlow}
		\|f\|_{s}(t) \lesssim \|f\|_{H^2(\mathbb{R})}(t) \lesssim 1,
\end{equation}
where in the above inequality we can allow $-\frac{1}{2} < s < \frac{3}{2}$.   From the argument in  \cite{CCGRPS} analogous to \eqref{sobolevuniform}, it can be shown that for any $0 < \delta < \frac{1}{2}$ we have
\bea\label{sobolevuniform2D}
\|f\|_{H^{l}(\mathbb{R})}(t) \leq \|f_{0}\|_{H^{l}(\mathbb{R})}\exp(CP(c_{0})\|f_{0}\|_{1+\delta}),
\quad l \ge 2.
\eea
Then the uniform bound of
$\|f\|_{H^2(\mathbb{R})}(t) \lesssim 1$ follows from \eqref{sobolevuniform2D} using the embedding 
\eqref{hembed} as in \eqref{sbound2Dlow} on the norm $\|f_{0}\|_{1+\delta}$.

Following the proof of \eqref{1normineq2D} that is given in \cite{JEMS} it can be directly shown that
\bea\label{1S2Ddiffbound}
\frac{d}{dt}\|f\|_{s}(t) \leq -C \|f\|_{s+1},  \quad 0 \le s \le 1.
\eea
Now using Lemma \ref{decaylemma}, \eqref{1S2Ddiffbound} for $\ss =s \in [0, 1]$ and  \eqref{sbound2Dlow} we obtain 
	\bea\label{initial2Ddecay}
	\|f\|_{s} \lesssim (1+t)^{-s+\rr},
	\eea
	for any $-\frac{1}{2} < \rr < s$; here the implicit constant depends on $\|f_{0}\|_{\rr}$.  

The next step is to obtain uniform in time bounds for $\|f\|_{s}(t)$ when $-1 < s \le -\frac{1}{2}$.

\begin{prop}
Suppose $f$ is the solution to the Muskat problem \eqref{interfaceq2D} in 2D  described by Theorem \ref{2Doldthm} with $\|f_{0}\|_{s}< \infty$ for some $-1 < s < -\frac{1}{2}$.   Then 
\bea\label{decay2d}
\frac{d}{dt} \|f \|_{s} (t) \lesssim \|f\|_{1},
\eea
where the implicit constant depends on $s$, $c_{0}$ and $\|f_0\|_{H^{2}(\mathbb{R})}$.
\end{prop}

\begin{proof}
The proof follows similarly to the proof of \eqref{decay1}. The range of $s$ allowed is different due to range of acceptable exponents allowed by the Hardy-Littlewood-Sobolev inequality in one dimension.

Similarly to the proof in the 3D case, we have
$$\frac{d}{dt} \|f\|_{s} (t)+ \int_{\mathbb{R}} d\xi \ |\xi|^{s+1}|\hat{f}(\xi)|  \leq  \int_{\mathbb{R}} d\xi \ |\xi|^{s}|\mathscr{F}(N(f))(\xi)|.$$
From the proof of Theorem 3.1 in \cite{JEMS}, we obtain the inequality:
$$ \int_{\mathbb{R}} d\xi \ |\xi|^{s}|\mathscr{F}(N(f))(\xi)| \leq 2 \sum_{n\geq 1}   \int_{\mathbb{R}}  |\xi_{2n}||\hat{f}(\xi_{2n})| \Big(|\cdot|^{s} \ast \Big(\ast^{2n}|\cdot||\tilde{\hat{f}}(\cdot)| \Big) \Big)(\xi_{2n}) d\xi_{2n}. $$ From here, we can apply the Hardy-Littlewood-Sobolev inequality and Young's inequality to obtain, as in the 3D case, that
\begin{multline*}
	\int_{\mathbb{R}} d\xi \ |\xi|^{s}|\mathscr{F}(N(f))(\xi)| 
	\\
\leq 2 C_{s} \||\xi|\hat{f}(\xi)\|_{L^{2}}^{2-\theta}\|f\|_{1}^{\theta}\sum_{n\geq 1} \|f\|_{1}^{2n-1} = 2 C_{s} \|f\|_{H^{2}}^{2-\theta}\|f\|_{1}^{1+\theta}\sum_{n\geq 0} \|f\|_{1}^{2n},
\end{multline*}
where we used the fact that $\||\xi|\hat{f}(\xi)\|_{L^{2}} \leq \|f\|_{H^{2}(\mathbb{R})}$ by Plancharel's identity. By the uniform bounds on $\|f\|_{H^{2}(\mathbb{R})}$ and $\|f\|_{1}$, we obtain the result.  
\end{proof}

Then using \eqref{decay2d} combined with \eqref{initial2Ddecay}, analogous to Proposition \ref{prop10} we obtain
\begin{equation}\label{uniform2ds}
\|f\|_{s}(t) \lesssim 1,	
\end{equation}
which now holds uniformly in time for $-1<s\leq \frac{3}{2}$.  The uniform bound  when $s=-\frac{1}{2}$ is obtained using the argument from \eqref{interpbd}.   Further analogous to \eqref{decayepsilon} using \eqref{1S2Ddiffbound} and  Lemma \ref{decaylemma} we conclude the time decay
\bea\label{decay2Depsilon}
\|f\|_{s}(t) \lesssim (1+t)^{-s+\rr},
\eea
which now holds for any $s \in [0, 1]$ and any $\rr \in (-1, s)$ where the implicit constant in particular depends on $\|f\|_{\rr}$ and $c_{0}$.  

Lastly, for the critical case, analogous to Proposition \ref{endpointprop} using  \eqref{decay2Depsilon} we can show 
$$
\|f\|_{-1,\infty}(t)\lesssim 1,
$$ 
where the implicit constant depends on $\|f_{0}\|_{-1,\infty} < \infty$ and $c_0$.  This bound enables us to analogously prove the end point decay rate of \eqref{decay2Depsilon} with $\rr=-1$.  Also the 2D version of Proposition \ref{higherdecay3D} follows similarly.  Collecting all of these estimates establishes Theorem \ref{mainthm} in the 2D case. \hfill {\bf Q.E.D.}

\end{document}